\documentclass[draft]{amsart}
\usepackage[utf8]{inputenc}
\usepackage{a4wide}
\usepackage{amsfonts}
\usepackage{amssymb}
\usepackage{latexsym}
\usepackage{color}
\usepackage{euscript}  
\usepackage{setspace}
\usepackage{setspace}
\usepackage{mathrsfs}

\numberwithin{equation}{section} \theoremstyle{plain}
\makeatletter
\@namedef{subjclassname@2010}{%
\textup{2010} Mathematics Subject Classification}
\makeatother

\newtheorem{thm}{Theorem}[section]
\newtheorem{prop}[thm]{Proposition}
\newtheorem{cor}[thm]{Corollary}
\newtheorem{lem}[thm]{Lemma}

\theoremstyle{definition}

\newtheorem{exa}[thm]{{\it Example}}
\theoremstyle{remark}
\newtheorem{rem}[thm]{{\it Remark}}

\DeclareMathOperator{\dom}{D}
\DeclareMathOperator{\lin}{\mbox{\sc lin}}
\DeclareMathOperator{\D}{d}

\newcommand*{\borel}[1]{\mathfrak{B}(#1)}
\newcommand*{\card}[1]{\mathrm{card}(#1)}
\newcommand*{\is}[2]{\langle#1,#2\rangle}

\newcommand*{\cbb}{\mathbb C}

\newcommand*{\zbb}{\mathbb Z}
\newcommand*{\rbb}{\mathbb R}
\newcommand*{\nbb}{\mathbb N}
\newcommand*{\rbop}{{\overline{\mathbb R}_+}}

\newcommand*{\varthetab}{\boldsymbol \vartheta}
\newcommand*{\lambdab}{\boldsymbol \lambda}
\newcommand*{\lambdawb}{ \widetilde{\boldsymbol\lambda}}
\newcommand*{\varLambdab}{\boldsymbol \varLambda}
\newcommand*{\varXib}{\boldsymbol \varXi}
\newcommand*{\varGammab}{\boldsymbol \varGamma}
\newcommand*{\varDeltab}{\boldsymbol \varDelta}
\newcommand*{\hhb}{\boldsymbol{\mathcal{H}}}
\newcommand*{\kkb}{\boldsymbol{\mathcal{K}}}
\newcommand*{\bsf}{\boldsymbol{f}}
\newcommand*{\bsg}{\boldsymbol{g}}
\newcommand*{\bsE}{\boldsymbol{E}}
\newcommand*{\bsF}{\boldsymbol{F}}

\newcommand*{\varThetab}{\boldsymbol{\varTheta}}
\newcommand*{\varPsib}{\boldsymbol{\varPsi}}

\newcommand*{\hh}{\mathcal H}
\newcommand*{\ff}{\mathcal F}
\newcommand*{\ee}{\mathcal E}
\newcommand*{\kk}{\mathcal K}

\newcommand*{\cfl}{C_{\phi, \varLambdab}}
\newcommand*{\cflx}{C_{\phi,\varLambdab,x}}

\newcommand*{\ascr}{{\mathscr A}}
\newcommand*{\bscr}{{\mathscr B}}

\newcommand*{\mscr}{{\mathscr M}}
\newcommand*{\fscr}{{\mathscr F}}

\def\bsl{{\boldsymbol L}}
\def\bsb{{\boldsymbol B}}

\newcommand*{\gsf}{\mathsf{g}}
\newcommand*{\jsf}{\mathsf{j}}
\newcommand*{\Gsf}{\mathsf{G}}
\newcommand*{\hsf}{{\mathsf h}}

\newcommand*{\Ge}{\geqslant}
\newcommand*{\Le}{\leqslant}

\begin{document}
\setstretch{1.2}
\title[Quasinormal extensions of subnormal operator-weighted composition operators]{Quasinormal extensions of subnormal operator-weighted composition operators in $\ell^2$-spaces}
\author[P.\ Budzy\'{n}ski]{Piotr Budzy\'{n}ski}
\address{Katedra Zastosowa\'{n} Matematyki, Uniwersytet Rolniczy w Krakowie, ul. Balicka 253c, 30-189 Krak\'ow, Poland}
\email{piotr.budzynski@ur.krakow.pl}

\author[P.\ Dymek]{Piotr Dymek}
\address{Katedra Zastosowa\'{n} Matematyki, Uniwersytet Rolniczy w Krakowie, ul. Balicka 253c, 30-189 Krak\'ow, Poland}
\email{piotr.dymek@ur.krakow.pl}

\author[A.\ P\l aneta]{Artur P\l aneta}
\address{Katedra Zastosowa\'{n} Matematyki, Uniwersytet Rolniczy w Krakowie, ul. Balicka 253c, 30-189 Krak\'ow, Poland}
\email{artur.planeta@ur.krakow.pl}

\begin{abstract}
We prove the subnormality of an operator-weighted composition operator whose symbol is a transformation of a discrete measure space and weights are multiplication operators in $L^2$-spaces under the assumption of existence of a family of probability measures whose Radon-Nikodym derivatives behave regular along the trajectories of the symbol. We build the quasinormal extension which is a weighted composition operator induced by the same symbol. We give auxiliary results concerning commutativity of operator-weighted composition operators with multiplication operators.
\end{abstract}
\maketitle
\section{Introduction}
Recent years have brought rapid development in studies over unbounded composition operators in $L^2$-spaces (see \cite{b-jfsa-2012, b-d-j-s-aaa-2013,  b-d-p-f-2016, b-d-p-bjma-2016, b-j-j-s-ampa-2014, b-j-j-s-jmaa-2014, b-j-j-s-jfa-2015, b-j-j-s-golf, c-h-jot-1993, j-mpcps-2003, j-s-pems-2001}) and weighted shifts on directed trees (see \cite{b-j-j-s-jmaa-2012, b-j-j-s-jmaa-2013, b-j-j-s-jmaa-2016, j-j-s-mams-2012, j-j-s-jfa-2012, j-j-s-caot-2013, j-j-s-pams-2014,p-jmaa-2016, t-jmaa-2015}), mostly in connection with the question of their subnormality. One can see that all these operators belong to a larger class of Hilbert space operators composed of unbounded weighted composition operators in $L^2$-spaces (see \cite{b-j-j-s-wco}). The class, in the bounded case, is well-known and there is extensive literature  concerning properties of its members, both in the general case as well as in the case of particular realizations like weighted shifts, composition operators, or multiplication operators (see, e.g., \cite{shi, s-m}). Weighted composition operators acting in spaces of complex-valued functions have a natural generalization in the context of vector-valued functions -- the usual complex-valued weight function is replaced by a function whose values are operators. These are the operators we call operator-weighted composition operators, we will refer to them as {\em o-wco's}. Their particular realizations are weighted shifts acting on $\ell^2$-spaces of Hilbert space-valued functions or composition operators acting on Hilbert spaces of vector-valued functions, which has already been studied (see, e.g., \cite{h-j-pams-1996, j-gmj-2004, k-pams-1984, l-bams-1971}). Interestingly, many weighted composition operators can be represented as operator weighted shifts (see \cite{c-j-jmaa-1991}). 

In this paper we turn our attention to o-wco's that act in an $\ell^2$-space of $L^2$-valued functions and have a weight function whose values are multiplication operators. We focus on the subnormality of these operators. Our work is motivated by the very recent criterion (read: sufficient condition) for the subnormality of unbounded composition operators in $L^2$-spaces (see \cite{b-j-j-s-jfa-2015}). The criterion relies on a construction of quasinormal extensions for composition operators, which is doable if the so-called consistency condition holds. It turns out that these ideas can be used in the context of o-wco's leading to the criterion for subnormality in Theorem \ref{kryterium}, which is the main result of the paper. The quasinormal extension for a composition operator built as in \cite{b-j-j-s-jfa-2015} is still a composition operator which acts over a different measure space. Interestingly, in our case for a given o-wco we get a quasinormal extension which also is an o-wco and acts over the same measure space. The extension comes from changing the set of values of a weight function. We investigate the subnormality of o-wco's in the bounded case and we show in Theorem \ref{konieczny} that the conditions appearing in our criterion are not only sufficient but also necessary in this case. Later we provide a few illustrative examples. The paper is concluded with some auxiliary results concerning commutativity of wco's and multiplication operators. 
\section{Preliminaries}
In all what follows $\zbb$, $\rbb$ and $\cbb$ stands for the sets of integers, real numbers and complex numbers, respectively; $\nbb$, $\zbb_+$ and $\rbb_+$ denotes the sets of positive integers, nonnegative integers and nonnegative real numbers, respectively. Set $\rbop = \rbb_+ \cup \{\infty\}$. If $S$ is a set and $E\subseteq S$, then $\chi_E$ is the characteristic function of $E$. Given a $\sigma$-algebra  $\varSigma$ of subsets of $S$ we denote by $\mscr(\varSigma)$ the space of all $\varSigma$-measurable $\cbb$- or $\rbop$-valued (depending on the context) functions on $S$; by writing $\mscr_+(\varSigma)$ we specify that the functions have values in $\rbop$. If $\mu$ and $\nu$ are positive measures on  $\varSigma$ and $\nu$ is absolute continuous with respect to $\mu$, then we denote this fact by writing $\nu\ll\mu$. If $Z$ is a topological space, then $\borel{Z}$ stands for the $\sigma$-algebra of all Borel subsets of $Z$. For $t\in \rbb_+$ the symbol $\delta_t$ stands for the Borel probability measure on $\rbb_+$ concentrated at $t$. If $\hh$ is a Hilbert space and $\ff$ is a subset of $\hh$, then $\lin \ff$ stands for the linear span of $\ff$.

Let $\hh$ and $\kk$ be Hilbert spaces (all Hilbert spaces considered in this paper are assumed to be complex). Then $\bsl(\hh,\kk)$ stands for the set of all linear (possibly unbounded) operators defined in a Hilbert space $\hh$ with values in a Hilbert space $\kk$. If $\hh=\kk$, then we write $\bsl(\hh)$ instead of $\bsl(\hh,\hh)$. $\bsb(\hh)$ denotes the algebra of all bounded linear operators with domain equal to $\hh$. Let $A\in \bsl(\hh,\kk)$. Denote by $\dom(A)$, $\overline{A}$ and $A^*$ the domain, the closure and the adjoint of $A$ (in case they exist). A subspace $\ee$ of $\hh$ is a core of $A$ if $\ee$ is dense in $\dom(A)$ in the graph norm $\|\cdot\|_A$ of $A$; recall that $\|f\|_A^2:=\|Af\|^2+\|f\|^2$ for $f\in\dom(A)$. If $A$ and $B$ are operators in $\hh$ such that $\dom(A)\subseteq \dom(B)$ and $Af=Bf$ for every $f\in\dom(A)$, then we write $A\subseteq B$. A closed densely defined operator $N$ in $\hh$ is said to be {\em normal} if $N^*N=NN^*$. A densely defined operator $S$ in $\hh$ is said to be {\em subnormal} if there exists a Hilbert space $\kk$ and a normal operator $N$ in $\kk$ such that $\hh \subseteq \kk$ (isometric embedding) and $Sh = Nh$ for all $h \in \dom(S)$. A closed densely defined operator $A$ in $\hh$ is {\em quasinormal} if and only if $U|A|\subseteq |A|U$, where $|A|$ is the modulus of $A$ and $A=U|A|$ is the polar decomposition of $A$. It was recently shown (cf. \cite{j-j-s1}) that:
\begin{align} \label{QQQ}
\begin{minipage}{85ex}
A closed densely defined operator $Q$ is quasinormal if and only if $Q Q^*Q = Q^* Q Q$.
\end{minipage}
\end{align}
It is well-known that quasinormal operators are subnormal (cf.\ \cite{bro,sto-sza2}).

Throughout the paper $X$ stands for a countable set. Let $\mu$ be a discrete measure on $X$, i.e., $\mu$ is a measure on $2^X$, the power set of $X$, such that $0<\mu_x:=\mu(\{x\})<\infty$ for every $x\in X$. Let $\hhb=\{\hh_x\colon x\in X\}$ be a family of (complex) Hilbert spaces. Then $\ell^2(\hhb, \mu)$ denotes the Hilbert space of all sequences ${\boldsymbol{f}}=\{f_x\}_{x\in X}$ such that $f_x\in\hh_x$ for every $x\in X$ and $\sum_{x\in X}||f_x||_{\hh_x}^2\mu_x<\infty$. For brevity, if this leads to no confusion, we suppress the dependence of the norm $\|\cdot\|_{\hh_x}$ on $\hh_x$ and write just $\|\cdot\|$. If $\mu$ is the counting measure on $X$, then we denote $\ell^2(\hhb,\mu)$ by $\ell^2(\hhb)$.

Here and later on we adhere to the notation that all the sequences, families or systems indexed by a set $X$ will be denoted by bold symbols while the members will be written with normal ones.

Let $X$ be a countable set, $\mu$ be a discrete measure on $X$, $\phi$ be a self-map of $X$,
$\hhb=\{\hh_x\colon x\in X\}$ be a family of Hilbert spaces and $\varLambdab=\{\varLambda_x\colon x\in X\}$ be a family of operators such that $\varLambda_x\in\bsl(\hh_{\phi(x)},\hh_x)$ for every $x\in X$. We say that $(X,\phi,\mu,\hhb,\varLambdab)$ is {\em admissible} then. By saying that $(X,\phi,\hhb,\varLambdab)$ is admissible we mean that $(X,\phi,\nu,\hhb,\varLambdab)$ is admissible where $\nu$ is the counting measure. Denote by $\dom_{\varLambdab}$ the set of all $\boldsymbol{f}\in \ell^2(\hhb,\mu)$ such that $f_y\in\bigcap_{z\in\phi^{-1}(\{y\})}\dom(\varLambda_z)$ for every $y\in\phi(X)$, i.e., $f_{\phi(x)}\in \dom(\varLambda_x)$ for every $x\in X$. Then an {\em operator-weighted composition operator} in $\ell^2(\hhb,\mu)$ induced by $\phi$ and $\varLambdab$
is the operator
    \begin{align*}
    \cfl \colon \ell^2(\hhb,\mu)\supseteq\dom\big(\cfl\big)\to \ell^2(\hhb,\mu)
    \end{align*}
defined according to the following formula
\allowdisplaybreaks
    \begin{gather*}
    \dom\big(\cfl\big)=\Big\{\bsf\in\dom_{\varLambdab} \colon \quad\sum_{x\in X} \big\|\varLambda_x f_{\phi(x)}\big\|_{\hh_x}^2\mu_x<\infty \Big\},\\
    \big(\cfl \bsf\big)_x= \varLambda_{x}f_{\phi(x)} ,\quad x\in X,\ \bsf \in \dom\big(\cfl\big).
   \end{gather*}
\begin{rem}\label{wco}
In case every $\hh_x$ is equal to $\cbb$ and $\varLambda_x$ is just multiplying by a complex number $w_x$, $\cfl$ is the classical {\em weighted composition operator} $C_{\phi, w}$ induced by $\phi$ and the weight function $w\colon X\to \cbb$ given by $w(x)=w_x$, and acting in the $L^2(\mu):=L^2(X,2^X,\mu)$. More precisely, $C_{\phi, w}\colon L^2(\mu)\supseteq\dom(C_{\phi, w})\to  L^2(\mu)$ is defined by
\begin{gather*}
\dom(C_{\phi, w})=\big\{f\in L^2(\mu) \colon w\cdot(f\circ\phi)\in L^2(\mu) \big\},\\ 
C_{\phi, w} f=w\cdot (f\circ\phi),\quad f\in \dom(C_{\phi, w}).
\end{gather*}
For more information on (unbounded) weighted composition operators in $L^2$-spaces we refer the reader to \cite{b-j-j-s-wco}.
\end{rem}
It is clear that the operator $U\colon \ell^2(\hhb,\mu)\to \ell^2(\hhb)$ given by
\begin{align}\label{kawa}
(U\bsf)_x=\sqrt{\mu_x}f_x,\quad x\in X,\ \bsf\in\ell^2(\hhb,\mu), \end{align}
is unitary. Using this operator one can show that the operator $\cfl$ in $\ell^2(\hhb,\mu)$ is in fact unitarily equivalent to an o-wco acting in $\ell^2(\hhb)$.
\begin{prop}\label{liczaca}
Let $(X,\phi,\mu,\hhb,\varLambdab)$ be admissible. Then $\cfl$ in $\ell^2(\hhb,\mu)$ is unitarily equivalent to $C_{\phi,\varLambdab^\prime}$ in $\ell^2(\hhb)$ with $\varLambdab^\prime=\Big\{\sqrt{\frac{\mu_x}{\mu_{\phi(x)}}}\varLambda_x \colon x\in X\Big\}$ via $U$ defined by \eqref{kawa}.
\end{prop}
The above enables us to restrict ourselves to studying o-wco's in exclusively in $\ell^2(\hhb)$ in all considerations that follow. 

The following convention is used in the paper: if $(X,\hhb, \phi,\varLambdab)$ is admissible, $\mathcal{P}$ is a property of Hilbert space operators, then we say that $\varLambdab$ satisfies $\mathcal{P}$ if and only if $\varLambda_x$ satisfies $\mathcal{P}$ for every $x\in X$.
\begin{lem}\label{closed}
Let $(X,\phi,\hhb,\varLambdab)$ be admissible. Then $\cfl$ in $\ell^2(\hhb)$ is closed whenever $\varLambdab$ is closed.
\end{lem}
\begin{proof}
Suppose that $\varLambda_x$ is closed for every $x\in X$. Take a sequence $\{\bsf^{(n)}\}_{n=1}^\infty\subseteq \dom(\cfl)$ such that $\bsf^{(n)}\to \bsf\in\ell^2(\hhb)$ and $\cfl \bsf^{(n)}\to \bsg\in\ell^2(\hhb)$ as $n\to\infty$. Clearly, for every $x\in X$, $f^{(n)}_x\to f_x$ and $\big(\cfl \bsf^{(n)}\big)_x\to g_x$ as $n\to\infty$. The latter implies that $\varLambda_{x}f^{(n)}_{\phi(x)}\to g_x$ as $n\to \infty$ for every $x\in X$. Since all $\varLambda_x$'s are closed we see that for every $x\in X$, $f_{\phi(x)}\in\dom(\varLambda_{x})$ and $\varLambda_{x} f_{\phi(x)}=g_x$. This yields $\bsf\in\dom(\cfl)$ and $\cfl \bsf=\bsg$.
\end{proof}
The reverse implication does not hold in general.
\begin{exa}
Let $X=\{-1,0,1\}$. Let $\phi\colon X\to X$ be the transformation given by $\phi(-1)=\phi(1)=0$ and $\phi(0)=0$. Let $\hhb=\{\hh_x\colon x\in X\}$ be a set of Hilbert spaces such that $\hh_1\subseteq \hh_{-1}$, and $\varLambdab=\big\{\varLambda_x \in \bsl(\hh_0,\hh_x)\colon  x\in X\big\}$ be a set of operators such that $\varLambda_0$ and $\varLambda_{1}$ are closed, $\varLambda_{-1}$ is not closed, and $\varLambda_1\subseteq\varLambda_{-1}$. Then $(X,\phi,\hhb,\varLambdab)$ is admissible. Let $\cfl$ be the o-wco in $\ell^2(\hhb)$ induced by $\phi$ and $\varLambdab$. Let $\{\bsf^{(n)}\}_{n=1}^\infty$ be a sequence in $\dom(\cfl)$ such that $\bsf^{(n)}\to \bsf\in\ell^2(\hhb)$ and $\cfl \bsf^{(n)}\to \bsg\in\ell^2(\hhb)$ as $n\to\infty$. As in the proof of Lemma \ref{closed} we see that $\varLambda_{x}f^{(n)}_0\to g_x$ for every $x\in X$. Since $\varLambda_1\subseteq\varLambda_{-1}$ we get $g_{-1}=g_1$. Closedness of $\varLambda_1$ implies that $f_0\in \dom(\varLambda_1)\subseteq\dom(\varLambda_{-1})$ and $\varLambda_{-1} f_0=\varLambda_1 f_0=g_1=g_{-1}$. For $x=0$ we can argue as in the proof of Lemma \ref{closed} to show that $f_0\in\dom(\varLambda_0)$ and $\varLambda_xf_0=g_0$. These facts imply that $f\in\dom(\cfl)$ and $\cfl f=g$, which shows that $\cfl$ is closed. 
\end{exa}
%
Some properties of the operator $\cfl$ can be deduced by investigating operators 
\begin{align*}
\cflx\colon \hh_x\supseteq \dom(\cflx) \to \ell^2\big(\hhb^x\big),\quad x\in \phi(X),
\end{align*}
with $\hhb^x=\{\hh_y\colon y\in\phi^{-1}(\{x\})\}$, which are defined by
\begin{gather*}
\dom(\cflx) = \Big\{f\in \hh_x\colon f\in \bigcap_{y\in\phi^{-1}(\{x\})}\dom(\varLambda_{y})\text{ and } \sum_{y\in\phi^{-1}(\{x\})}\|\varLambda_{y}f\|^2<\infty\Big\}\\
\big(\cflx f\big)_y=\varLambda_{y}f,\quad y\in\phi^{-1}(\{x\}),\ f\in\dom(\cflx).
\end{gather*}
\begin{prop}\label{ortsumdense}
Let $(X,\phi,\hhb,\varLambdab)$ be admissible. Then $\cfl$ is a densely defined operator in $\ell^2(\hhb)$ if and only if  $\cflx$ is a densely defined operator for every $x\in \phi(X)$.
\end{prop}
\begin{proof}
Suppose that for every $x\in \phi(X)$, $\cflx$ is densely defined while $\cfl$ is not. Then there exists $\bsf\in\ell^2(\hhb)$ and $r>0$ such that $\mathbb{B}(\bsf,r)\cap \dom(\cfl)=\varnothing$, where $\mathbb{B}(\bsf,r)$ denotes the open ball in $\ell^2(\hhb)$ with center $\bsf$ and radius $r$. Since all $\cflx$'s are densely defined, we may assume that $f_x\in\dom(\cflx)$ for every $x\in X$. From $\sum_{x\in X}\|f_x\|^2<\infty$ we deduce that there exists a finite set $Y\subseteq X$ such that $\sum_{x\in X\setminus Y}\|f_x\|^2<\frac{r^2}{2}$. Then $\bsg\in\ell^2(\hhb)$ given by $g_x=\chi_Y(x) f_x$ belongs to $\mathbb{B}(\bsf,r)$. Moreover, $\sum_{x \in X} \| \varLambda_x g_{\phi(x)}\|^2 \Le  \sum_{x \in X} \| \varLambda_x f_{\phi(x)}\|^2 < \infty$, which shows that $\bsg \in \dom(\cfl)$. This contradiction proves the ''if'' part. The ''only if'' part follows from the fact that $f_x \in \dom(\cflx)$ for every $x \in \phi(X)$ whenever $ \bsf \in \dom(\cfl)$.
\end{proof}
\begin{rem}
It is worth mentioning that if $\varLambdab$ is closed, then $\cfl$ is unitarily equivalent to the orthogonal sum of operators $\cflx$, $x\in X$ (this can be shown by using a version of \cite[Theorem 5, p. 81]{b-s}). This leads to another proof of Proposition \ref{ortsumdense} in case $\varLambdab$ is closed.
\end{rem}
In the course of our study we use frequently multiplication operators. Below we set the notation and introduce required terminology concerning these operators. Suppose $\{(\varOmega_x,\ascr_x,\mu_x)\colon x\in X\}$ and $\{(\varOmega_x,\ascr_x,\nu_x)\colon x\in X\}$ are families of $\sigma$-finite measure spaces. Let $\varGammab=\{\varGamma_x\colon x\in X\}$ with $\varGamma_x\in\mscr(\ascr_x)$ for $x\in X$. Assume that $|\varGamma_x|^2\nu_x\ll\mu_x$ for every $x\in X$. Let $\hhb=\{L^2(\mu_x)\colon x\in X\}$ and $\kkb=\{L^2(\nu_x)\colon x\in X\}$. Then $M_{\varGammab}\colon \ell^2(\hhb)\ni\dom(M_{\varGammab})\to \ell^2(\kkb)$, {\em the operator of multiplication by $\varGammab$}, is given by
\begin{gather*}
\dom(M_{\varGammab})=\big\{\bsf\in \ell^2(\hhb)\colon \varGamma_yf_y\in \kk_y \text { for every } y\in X \text{ and } \sum_{x\in X}\|\varGamma_x f_x\|_{\kk_x}^2<\infty\big\},\\
\big(M_{\varGammab} \bsf\big)_x = \varGamma_x f_x,\quad x\in X,\ \bsf\in \dom(M_{\varGammab}).
\end{gather*}
Clearly, $M_{\varGammab}$ is well-defined. Note that the above definition agrees with the definition of classical multiplication operators if $X$ is a one-point set $\{x_0\}$ and $\mu_{x_0}=\nu_{x_0}$. 
Below we show when a multiplication operator $M_{\varGammab}$ is closed (since our setting in not entirely classical we give a short proof).
\begin{lem} \label{poswietach-1}
Let $x \in X$. If $\frac{\D|\varGamma_x|^2\nu_x}{\D\mu_x}<\infty$ a.e. $[\mu_x]$, then $M_{\varGamma_x}\colon L^2(\mu_x) \to L^2(\nu_x)$ is densely defined and closed.
\end{lem}
\begin{proof}
Set $h_x:=\frac{\D|\varGamma_x|^2\nu_x}{\D\mu_x}$. Applying the Radon-Nikodym theorem we obtain
\begin{align}\label{dzmnozenie}
\dom(M_{\varGamma_x})=L^2\big((1+h_x) \D\mu_x\big).
\end{align}
Since, by \cite[Lemma 12.1]{b-j-j-s-ampa-2014}, $L^2\big((1+h_x) \D\mu_x\big)$ is dense in $L^2(\mu_x)$, \eqref{dzmnozenie} implies that $M_{{\varGamma_x}}$ is densely defined. For every $f\in\dom(M_{\varGamma_x})$ the graph norm of $f$ equals $\int_{\varOmega_x} |f|^2(1+h_x)\D\mu_x$. Thus $\dom(M_{\varGamma_x})$ is complete in the graph norm of $M_{\varGamma_x}$, which proves that $M_{\varGamma_x}$ is closed (see \cite[Theorem 5.1]{wei}). 
\end{proof}
It is easily seen that $M_{\varGammab}$ is the orthogonal sum of $M_{\varGamma_y}$'s, which yields the aforementioned fact.
\begin{cor}\label{poswietach}
Suppose $\frac{\D|\varGamma_y|^2\nu_y}{\D\mu_y}<\infty$ a.e. $[\mu_y]$ for all $y\in X$. Then $M_{\varGammab}$ is densely defined and closed.
\end{cor}
It is well-known that classical multiplication operator is selfadjoint if the multiplying function is $\overline{\rbb}$-valued. One can show that the same applies to $M_{\varGammab}$, i.e., $M_{\varGammab}$ is selfadjoint if $\varGammab$ is $\overline{\rbb}$-valued (of course, assuming dense definiteness of $M_{\varGammab}$).
\section{The criterion}
In this section we provide a criterion for the subnormality of an o-wco $\cfl$ with  $\varLambdab$ being a family of multiplication operators acting in a common $L^2$-space. More precisely, we will assume that 
\begin{align}\label{zemanek}
\begin{minipage}{88ex}
$X$ is a countable set, $\phi$ is a self-map of $X$, $(W, \ascr, \varrho)$ is a $\sigma$-finite measure space, $\lambdab=\{\lambda_x\}_{x\in X}\subseteq \mscr(\ascr)$,   $\hhb=\{\hh_x\colon x\in X\}$, with $\hh_x=L^2(\varrho)$, and $\varLambdab=\{M_{\lambda_x}\colon x\in X\}$, where $M_{\lambda_x}\colon L^2(\varrho) \supseteq\dom(M_{\lambda_x})\to L^2(\varrho)$ is the operator of multiplication by $\lambda_x$.
\end{minipage}
\end{align}
The criterion for subnormality we are aiming for will rely on a well-known measure-theoretic construction of a measure from a measurable family of probability measures, which has already been used in the context of subnormality (see \cite{b-j-j-s-jfa-2015,lam-1988-mmj}). Using it we will build an extension for $\cfl$. To this end we consider a measurable space $(S,\varSigma)$ and a family $\{\vartheta_x^w\colon x\in X, w\in W\}$ of probability measures on $\varSigma$ satisfying the following conditions:
\begin{itemize}
    \item[(\texttt{A})] for all $x\in X$ and $\sigma\in\varSigma$ the map $W\ni w\mapsto \vartheta^w_x(\sigma)\in[0,1]$ is $\ascr$-measurable,
    \item[(\texttt{B})] for all $x\in X$ and $w\in W$, $|\lambda_x(w)|^2\vartheta_x^w\ll \vartheta_{\phi(x)}^w$.
    \end{itemize}
By $(\texttt{A})$, for every $x\in X$ the formula
    \begin{align}\label{miararoz}
    \widehat\varrho_x(A\times\sigma)=
    \int_W \int_S \chi_{A\times\sigma}(w,s)\D\vartheta^w_x(s)\D\varrho(w),
    \quad A\times\sigma\in\ascr\otimes\varSigma,
    \end{align}
where $\ascr\otimes\varSigma$ denotes the $\sigma$-algebra generated by the family $\{A\times\sigma\colon A\in \ascr, \sigma\in\varSigma\}$, defines a $\sigma$-finite measure $\widehat\varrho_x$ on $\ascr\otimes\varSigma$ (cf. \cite[Theorem 2.6.2]{ash}) which satisfies
    \begin{align}\label{normaroz}
    \int_{W\times S} F(w,s)\D\widehat\varrho_x(w,s)
        =
    \int_W \int_S F(w,s)\D\vartheta^w_x(s)\D\varrho(w), \quad F\in \mscr_+(\ascr\otimes\varSigma).
    \end{align}
We first show that the measures in the family $\{\widehat\varrho_x\colon x\in X\}$ satisfy similar absolute continuity condition as the measures in the family $\{\vartheta_x^w\colon x\in X, w\in W\}$, and that the Radon-Nikodym derivatives coming from the former family can be written in terms of the Radon-Nikodym derivatives coming from the latter one. For $x\in X$, let $\widehat\lambda_x\in\mscr(\ascr\otimes\varSigma)$ be given by $\widehat\lambda_x(w,s)=\lambda_x(w)$ for $(w,s)\in W\times S$.
\begin{lem}\label{warunekC}
Assume \eqref{zemanek}. Let $(S,\varSigma)$ be a measurable space and $\{\vartheta_x^w\colon x\in X, w\in W\}$ be a family of probability measures on $\varSigma$ satisfying conditions {\em (\texttt{A})} and {\em (\texttt{B})}. Then the following conditions hold$:$
\begin{itemize}
\item[(i)] for all $x\in X$, $|\widehat\lambda_x|^2\widehat\varrho_x\ll\widehat\varrho_{\phi(x)}$,
\item[(ii)] for every $x\in X$, $\varrho$-a.e. $w\in W$ and $\vartheta_{\phi(x)}^w$-a.e. $s\in S$, $\frac{\D|\widehat\lambda_x|^2\widehat\varrho_x}{\D\widehat\varrho_{\phi(x)}}(w,s)=\frac{\D|\lambda_x(w)|^2\vartheta_x^w}{\D\vartheta_{\phi(x)}^w}(s)$.
\end{itemize}
\end{lem}
\begin{proof}
Using $(\texttt{B})$ and \eqref{normaroz} we easily get (i). Now, for any $x\in X$ we define $H_x, h_x\colon W\times S\to \rbop$ by $H_x(w,s)=\frac{\D|\widehat\lambda_x|^2\widehat\varrho_x}{\D\widehat\varrho_{\phi(x)}}(w,s)$ and $h_x(w,s)=\frac{\D|\lambda_x(w)|^2\vartheta_x^w}{\D\vartheta_{\phi(x)}^w}(s)$. Then, for every $x\in X$, by the Radon-Nikodym theorem and \eqref{normaroz}, we have
\begin{align*}
\int_A\int_\sigma H_x(w,s)\D\vartheta_{\phi(x)}^w(s)\D\varrho(w)
&=\int_{A\times\sigma} H_x(w,s)\D\widehat\varrho_{\phi(x)}(w,s)\\
&=\int_A\int_\sigma |\lambda_x(w)|^2\D\vartheta_x^w(s)\D\varrho(w)\\
&=\int_A\int_\sigma h_x(w,s)\D\vartheta_{\phi(x)}^w(s)\D\varrho(w),\quad A\times\sigma\in\ascr\otimes \varSigma.
\end{align*}
This implies (ii), which completes the proof.
\end{proof}
We will use frequently the following notation
\begin{align}\label{l22}
\Gsf_x^w(s):=\left\{
            \begin{array}{ll}
              \sum_{y\in\phi^{-1}(\{x\})} \frac{\D|\widehat\lambda_y|^2\widehat\varrho_y}{\D\widehat\varrho_x}(w,s), & \hbox{for $x\in\phi(X)$,} \\
              0, & \hbox{otherwise.}
            \end{array}
          \right.
\end{align}
Note that for every $x\in X$ the function 
\begin{align}\label{tlustazupa}
\Gsf_x\colon W\times S\ni(w,s)\mapsto \Gsf_x^w(s)\in \rbop    
\end{align}
is $\ascr\otimes\varSigma$-measurable and, in view of Lemma \ref{warunekC}, we have
\begin{align}\label{obiad}
\sum_{y\in \phi^{-1}(\{x\})}\int_{W}\int_S |\lambda_y (w)|^2 |F(w,s)|^2 \D\vartheta_y^w(s)\D\varrho(w)=\int_{W\times S} \Gsf_x(w,s)|F(w,s)|^2\D\widehat\varrho_x(w,s)
\end{align}
for every $F\in \mscr(\ascr \otimes \Sigma)$ (here and later on we adhere to the convention that $\sum_{\varnothing}=0$). Set 
\begin{align*}
\Gsf=\{\Gsf_x\colon x\in X\}.
\end{align*} 
The following set of assumptions complements \eqref{zemanek}
\begin{align}\label{majdak}
\begin{minipage}{88ex}
$(S,\varSigma)$ is a measurable space, $\{\vartheta_x^w\colon x\in X, w\in W\}$ is a family of probability measures on $\varSigma$ satisfying $(\texttt{A})$ and $(\texttt{B})$, $\{\widehat\varrho_x\colon x\in X\}$ is a family of measures on $\ascr\otimes\varSigma$ given by \eqref{miararoz}, $\widehat\hhb=\{L^2(\widehat\varrho_x)\colon x\in X\}$, and $\widehat\varLambdab=\{M_{\widehat\lambda_x}\colon x\in X\}$, where $M_{\widehat\lambda_x}\colon L^2(\widehat\varrho_{\phi(x)})\supseteq\dom(M_{\widehat\lambda_x})\to L^2(\widehat\varrho_x)$ is the operator of multiplication by $\widehat\lambda_x$ given by $\widehat\lambda_x(w,s)=\lambda_x(w)$ for $(w,s)\in W\times S$. 
\end{minipage}
\end{align}
(That the operator $M_{\widehat\lambda_x}$ is well-defined follows from Lemma \ref{warunekC}.) 

It is no surprise that our construction leads to an extension of $\cfl$.
\begin{lem}\label{exten}
Assume \eqref{zemanek} and \eqref{majdak}. Then $\cfl\subseteq C_{\phi, \widehat\varLambdab}$.
\end{lem}
\begin{proof}
In view of \eqref{normaroz}, for every $x\in X$ the space $L^2(\varrho)$ can be isometrically embedded into $L^2(\widehat\varrho_x)$ via the mapping 
\begin{align*}
Q_x\colon L^2(\varrho)\ni f\mapsto F_x\in L^2(\widehat\varrho_x)   
\end{align*}
where 
\begin{align*}
\text{$F_x(w,s)=f(w)$ for $\widehat\varrho_x$-a.e. $(w,s)\in W\times S$.}
\end{align*}
Therefore, $\ell^2(\hhb)$ can be isometrically embedded into $\ell^2(\widehat\hhb)$ via $Q\in\bsl\big(\ell^2(\hhb), \ell^2(\widehat\hhb)\big)$ defined by $(Q \bsf)_x:= Q_x f_x$, $x\in X$, $\bsf\in\ell^2(\hhb)$. Since all the measures $\vartheta_x^w$ are probability measures, we deduce using \eqref{normaroz} that $\dom(\cfl)  \subseteq \dom(C_{\phi, \widehat\varLambdab} Q)$ and $Q \cfl \bsf= C_{\phi, \widehat\varLambdab}Q \bsf$ for every $\bsf\in \dom(\cfl)$. This means that $C_{\phi,\widehat\varLambdab}$ is an extension of $\cfl$.
\end{proof}
Next we formulate a few necessary results concerning properties of $C_{\phi,\widehat\varLambdab}$. First, we show its dense definiteness and closedness.
\begin{prop}\label{invitedlenie}
Assume \eqref{zemanek} and \eqref{majdak}. Suppose that for every $x\in X$, $\Gsf_x<\infty$ a.e.\ $[\widehat\varrho_{x}]$. Then $C_{\phi,\widehat\varLambdab}$ is a closed and densely defined operator in $\ell^2(\widehat\hhb)$.
\end{prop}
\begin{proof}
Let $x\in X$. Let $H_x=\frac{\D|\widehat \lambda_x|^2\widehat\varrho_x}{\D\widehat\varrho_{\phi(x)}}$.  Using \eqref{dzmnozenie} we get
\begin{align*}
\dom\big(\widehat\varLambda_x\big)= L^2\Big((1+H_x)\D\widehat\varrho_{\phi(x)}\Big).  
\end{align*}
This yields the equality
\begin{align*}
\dom(C_{\phi,\widehat\varLambdab, x})=L^2\big((1+\Gsf_x)\D\widehat\varrho_x\big),\quad x\in X.
\end{align*}
By \cite[Lemma 12.1]{b-j-j-s-ampa-2014}, the space $L^2\big((1+\Gsf_x)\D\widehat\varrho_x\big)$ is dense in $L^2(\widehat\varrho_x)$ and consequently $C_{\phi,\widehat\varLambdab, x}$ is densely defined. Since $x\in X$ can be chosen arbitrarily, by applying Proposition \ref{ortsumdense}, we get dense definiteness of $C_{\phi, \widehat\varLambdab}$. 

Now using Lemma \ref{closed} and Lemma \ref{poswietach-1},  we deduce that $C_{\phi, \widehat\varLambdab}$ is closed.
\end{proof}
The claim of the following auxiliary lemma is a direct consequence of $\sigma$-finiteness of $\widehat\varrho_x$. For the reader convenience we provide its proof.
\begin{lem} \label{swieta}
Assume \eqref{zemanek} and \eqref{majdak}. Let $x \in X$. Suppose that $\Gsf_x<\infty$ a.e.\ $[\widehat\varrho_x]$. Then there exists an $\ascr\otimes\varSigma$-measurable function $\alpha_x \colon W \times S \to (0,+\infty)$ such that
\begin{align}\label{alfa}
\int_{W\times S} \Big(1+\Gsf_x(w,s) +(\Gsf_x(w,s))^2\Big) (\alpha_x(w,s))^2 \D\widehat\varrho_x(w,s) < \infty.
    \end{align}
\end{lem}
\begin{proof}
Since the measure $\varrho$ is $\sigma$-finite there exists an $\ascr$-measurable function $f \colon W \to (0,+\infty)$ such that $\int_W f(w) \D \varrho(w)<\infty$. Now, for any given $w \in W$ and $s\in S$ we define
\begin{align*}
\alpha_x(w,s) =
\sqrt{\frac{f(w)}{1+\Gsf_x(w,s)+(\Gsf_x(w,s))^2}}.
\end{align*}
Clearly, the function $\alpha_x$ is
$\ascr\otimes\varSigma$-measurable. Moreover, by \eqref{normaroz}, we have 
\begin{multline*}
\int_{W\times S} \Big(1+\Gsf_x(w,s)+(\Gsf_x(w,s))^2\Big) (\alpha_x(w,s))^2 \D\widehat\varrho_x(w,s)\\
=\int_W \int_S f(w) \D\vartheta_x^w(s)\D\varrho(w)
=\int_W f(w)\D\varrho(w) < \infty.
\end{multline*}
This completes the proof.
\end{proof}
The proof of the criterion for subnormality of $\cfl$, which we give further below, relies much on the fact that $C_{\phi,\widehat\varLambdab}^*C_{\phi,\widehat\varLambdab}$ is a multiplication operator.
\begin{prop}\label{dziekan}
Assume \eqref{zemanek} and \eqref{majdak}. Suppose that for every $x\in X$, $\Gsf_x<\infty$ a.e.\ $[\widehat\varrho_{x}]$. Then $C_{\phi,\widehat\varLambdab}^* C_{\phi,\widehat\varLambdab} = M_{\Gsf}$.
\end{prop}
\begin{proof}
Let $x_0 \in X$, $A \in \ascr$, and $\sigma \in \Sigma$. Let $\bsE=\{E_x\colon x\in X\}$ be a family of functions $E_x\colon W\times S\to \rbb_+$ given by $E_x(w,s) = \chi_{A \times \sigma}(w,s) \alpha_{x_0}(w,s) \delta_{x, x_0}$, where function $\alpha_{x_0}$ satisfies \eqref{alfa}, and $\delta_{x,x_0}$ is the Kronecker delta. It follows from \eqref{alfa} that $\bsE\in\ell^2(\widehat\hhb)$. Moreover, by \eqref{normaroz}, \eqref{obiad}, and \eqref{alfa}, we get
\begin{align*}
\sum_{x \in X} \int_{W\times S} \Big|\widehat\lambda_x(w,s) &E_{\phi(x)} (w, s)\Big|^2\D\widehat\varrho_x(w,s)\\
&=\sum_{x \in \phi^{-1}(\{x_0\})} \int_A  \int_{\sigma} |\lambda_x(w)|^2 (\alpha_{x_0}(w,s))^2\D\vartheta_x^w(s)\D\varrho(w)\\
&\leqslant\int_{W\times S}\Gsf_{x_0}(w, s)(\alpha_{x_0}(w,s))^2\D\widehat\varrho_{x_0}(w,s) <\infty,
\end{align*}
which proves that $\bsE \in \dom(C_{\phi,\widehat\varLambdab})$.

Take now $\bsF\in\dom(C_{\phi, \widehat\varLambdab}^*C_{\phi, \widehat\varLambdab})$. Then
\begin{align*}
    \is{C_{\phi, \widehat\varLambdab}^*C_{\phi, \widehat\varLambdab}\bsF}{\bsE}=\int_{A\times \sigma}   \big(C_{\phi, \widehat\varLambdab}^*C_{\phi, \widehat\varLambdab}\bsF\big)_{x_0}(w,s)\,\alpha_{x_0}(w,s)\D\widehat\varrho_{x_0}(w,s).
\end{align*}
On the other hand, since $\bsE \in \dom(C_{\phi, \widehat\varLambdab})$, we have
\begin{align*}
\is{C_{\phi, \widehat\varLambdab}^*C_{\phi, \widehat\varLambdab}\bsF}{\bsE}
&=\is{C_{\phi, \widehat\varLambdab}\bsF}{C_{\phi, \widehat\varLambdab}\bsE}\\
&= \sum_{x\in X}\int_{W\times S} |\widehat \lambda_x(w,s)|^2 F_{\phi(x)}(w,s)\overline{E_{\phi(x)}(w,s)}\D\widehat\varrho_x(w,s)\\
&= \sum_{x\in \phi^{-1}(\{x_0\})}\int_{A\times \sigma} |\lambda_x(w)|^2 F_{x_0}(w,s) \alpha_{x_0}(w,s)\D\widehat\varrho_x(w,s)\\
&= \sum_{x\in \phi^{-1}(\{x_0\})}\int_{A\times \sigma} F_{x_0}(w,s) \alpha_{x_0}(w,s)\frac{\D |\widehat \lambda_x|^2\widehat \varrho_x}{\D \widehat \varrho_{x_0}} (w,s)\D\widehat\varrho_{x_0}(w,s)\\
&\overset{(\dag)}{=} \int_{A\times \sigma} \sum_{x\in \phi^{-1}(\{x_0\})}  F_{x_0}(w,s) \alpha_{x_0}(w,s)\frac{\D|\widehat \lambda_x|^2 \widehat \varrho_x}{\D \widehat \varrho_{x_0}}(w,s) \D\widehat\varrho_{x_0}(w,s)\\
&= \int_{A\times \sigma}  F_{x_0}(w,s)\Gsf_{x_0}(w,s)\alpha_{x_0}(w,s)\D\widehat\varrho_{x_0}(w,s),
\end{align*}
where in $(\dag)$ we used the fact that the function $(w,s) \mapsto \alpha_{x_0} (w,s) \Gsf_{x_0}(w,s) \in L^2(\widehat \varrho_{x_0})$ which means that the function $$(w,s)\mapsto \sum_{x\in \phi^{-1}(\{x_0\})}  \big|F_{x_0}(w,s)\big| \alpha_{x_0}(w,s)\frac{\D |\widehat \lambda_x|^2\widehat \varrho_x}{\D \widehat \varrho_{x_0}}(w,s) $$ belongs to $L^1(\widehat\varrho_{x_0})$. Since $A \in \ascr$, and $\sigma \in \Sigma$ can be arbitrarily chosen, we get
\begin{align*}
\big(C_{\phi, \widehat\varLambdab}^*C_{\phi, \widehat\varLambdab}\bsF\big)_{x_0}\alpha_{x_0}=F_{x_0}\Gsf_{x_0}\alpha_{x_0} \quad \text{for a.e.\ $[\widehat\varrho_{x_0}]$,}
\end{align*}
which implies that 
\begin{align*}
\big(C_{\phi, \widehat\varLambdab}^*C_{\phi, \widehat\varLambdab}\bsF\big)_{x_0}=F_{x_0}\Gsf_{x_0} \quad \text{for a.e.\ $[\widehat\varrho_{x_0}]$}.
\end{align*}
Thus $\bsF \in \dom(M_{\Gsf})$ and $C_{\phi, \widehat\varLambdab}^*C_{\phi, \widehat\varLambdab}\bsF = M_{\Gsf} \bsF$. In view of Proposition \ref{invitedlenie}, the operator $C_{\phi, \widehat\varLambdab}^*C_{\phi, \widehat\varLambdab}$ is selfadjoint (cf. \cite[Theorem 5.39.]{wei}). Since $M_{\Gsf}$ is selfadjoint as well, both the operators $C_{\phi, \widehat\varLambdab}^*C_{\phi, \widehat\varLambdab}$ and $M_\Gsf$ coincide. This completes the proof.
\end{proof}
After all the above preparations we are in the position now to prove the criterion for the subnormality of $\cfl$.
\begin{thm}\label{kryterium}
Assume \eqref{zemanek} and \eqref{majdak}. Suppose that for every $x\in X$, $\Gsf_x<\infty$ a.e.\ $[\widehat\varrho_{x}]$
and
\begin{align*}
\lambda_x\Gsf_{\phi(x)}=\lambda_x \Gsf_x\quad \text{a.e. $[\widehat\varrho_{x}]$ for every $x\in X$.}
\end{align*}
Then $C_{\phi, \widehat \varLambdab}$ is quasinormal. Moreover,  $\cfl$ is subnormal and  $C_{\phi, \widehat \varLambdab}$ is its quasinormal extension.
\end{thm}
\begin{proof}
Let $\bsF\in\ell^2(\widehat\hhb)$. By definition $\bsF\in\dom(C_{\phi,\widehat\varLambdab} M_{\Gsf})$ if and only if 
\begin{align}\label{ects1}
\sum_{x\in X} \int_{W\times S} |\Gsf_x(w,s)F_x(w,s)|^2\D \widehat  \varrho_x(w,s)<\infty    
\end{align}
and
\begin{align}\label{znowukasza}
\sum_{x\in X} \int_{W\times S} |\widehat\lambda_x (w,s)\Gsf_{\phi(x)}(w,s)F_{\phi(x)}(w,s)|^2 \D \widehat\varrho_x(w,s)<\infty.    
\end{align}
On the other hand, $\bsF\in\dom(M_{\Gsf}C_{\phi,\widehat\varLambdab})$ if and only if
\begin{align}\label{student}
\sum_{x\in X} \int_{W\times S} |\widehat\lambda_x(w,s)F_{\phi(x)}(w,s)|^2 \D \widehat\varrho_x(w,s)<\infty    
\end{align}
and
\begin{align}\label{ects2}
\sum_{x\in X} \int_{W\times S} |\widehat\lambda_x (w,s)\Gsf_{x}(w,s)F_{\phi(x)}(w,s)|^2\D\widehat\varrho_x(w,s)<\infty.    
\end{align}
Using the decomposition $X=\bigsqcup_{x\in X} \phi^{-1}(\{x\})$ and applying \eqref{normaroz} and \eqref{obiad} we see that \eqref{znowukasza} is equivalent to
\begin{align}\label{ects3}
\sum_{x\in X} \int_{W\times S} |\Gsf_x(w,s)|^3 |F_x(w,s)|^2\D\widehat\varrho_x(w,s)<\infty.
\end{align}
The same argument implies that \eqref{student} is equivalent to
\begin{align}\label{ects4}
\sum_{x\in X} \int_{W\times S} |\Gsf_x(w,s)| |F_x(w,s)|^2\D\widehat\varrho_x(w,s)<\infty.
\end{align}
Keeping in mind that $\bsF\in\ell^2(\widehat\hhb)$ and using \eqref{ects1}, \eqref{ects2}, \eqref{ects3}, and \eqref{ects4} we deduce that $\dom(C_{\phi,\widehat\varLambdab} M_{\Gsf})=\dom(M_{\Gsf} C_{\phi,\widehat\varLambdab})$. It is elementary to show that for every $\bsF\in \dom(M_{\Gsf} C_{\phi,\widehat\varLambdab})$, $M_{\Gsf} C_{\phi,\widehat\varLambdab} \bsF=C_{\phi,\widehat\varLambdab} M_{\Gsf}\bsF$. Therefore, $C_{\phi,\widehat\varLambdab}$ is quasinormal by \eqref{QQQ} and Theorem \ref{dziekan}. The ``moreover'' part of the claim follows immediately from Lemma \ref{exten} and the fact that operators having quasinormal extensions are subnormal (see \cite[Theorem 2]{sto-sza2}).
\end{proof}
\section{The bounded case}
In this section we investigate the subnormality of $\cfl$ under the assumption of boundedness of $\cfl$. We use a well-know relation between subnormality and Stieltjes moment sequences.

We begin with more notation. Suppose \eqref{zemanek} holds. Let $n\in\nbb$. Then $\varLambdab^{[n]}:=\{\varLambda_x^{[n]}\colon x\in X\}$, where $\varLambda_x^{[n]}:=M_{\lambda_x^{[n]}}\in \bsl(L^2(\varrho))$ with $\lambda_x^{[n]}:=\lambda_x\cdots \lambda_{\phi^{n-1}(x)}$, $x\in X$. We define a function 
\begin{align*}
    \hsf_x^{[n]} = \sum_{y \in \phi^{-n}(\{x\})} \Big|\lambda_y^{[n]}\Big|^2,\quad x\in X.
\end{align*}
We set $\lambda_x^{[0]}\equiv 1$, so that $\varLambda_x^{[0]}$ is the identity operator, and $\hsf_x^{[0]}\equiv 1$.

It is an easy observation that the $n$th power of $\cfl$ is the o-wco with the symbol $\phi^n$ and the weight $\varLambdab^{[n]}$. We state this below fact for future reference. 
\begin{lem}\label{potega}
Suppose \eqref{zemanek} holds. Let $n \in \nbb$. If $\cfl\in\bsb(\ell^2(\hhb))$, then $\cfl^n = C_{\phi^n,\varLambdab^{[n]}}$.
\end{lem}
The well-known characterization of subnormality for bounded operators due to Lambert (see \cite{l-jlms-1976}) states that an operator $A\in\bsb(\hh)$ is subnormal if and only if $\{\|A^n f\|^2\}_{n=0}^\infty$ is a Stieltjes moment sequence for every $f\in\hh$. Recall, that a sequence $\{a_n\}_{n=0}^\infty\subseteq\rbb_+$ is called a Stieltjes moment sequence if there exists a positive Borel measure $\gamma$ on $\rbb_+$ such that $a_n=\int_{\rbb_+} t^n\D\gamma(t)$ for every $n\in\zbb_+$. We call $\gamma$ a representing measure of $\{a_n\}_{n=0}^\infty$. If there exists a unique representing measure, then we say that $\{a_n\}_{n=0}^\infty$ is determinate. It is well-known that (see \cite{b-c-r, fug}; see also \cite[Theorem 3]{sza-1981-am}):
\begin{align}\label{sti}
\begin{minipage}{85ex}
A sequence $\{a_n\}_{n=0}^\infty\subseteq\rbb_+$ is a Stieltjes moment sequence if and only if
$$
\sum_{n,m=0}^\infty a_{n+m} \alpha(n)\overline{\alpha(m)}\Ge 0\text{ and } \sum_{n,m=1}^\infty a_{n+m+1} \alpha(n)\overline{\alpha(m)}\Ge 0,
$$
for every $\alpha\in\cbb^{(\zbb_+)}$, where $\cbb^{(\zbb_+)}$ denotes the set of all functions $\alpha\colon \zbb_+\to\cbb$ such that $\{k\in\zbb_+\colon \alpha(k)\neq 0\}$ is finite. Moreover, if  $\{a_n\}_{n=0}^\infty\subseteq\rbb_+$ is a Stieltjes moment sequence and there exists $r\in[0,\infty)$ such that 
$$a_{2n+2}\leqslant r^2a_{2n},$$ 
then $\{a_n\}_{n=0}^\infty$ is determinate and its representing measure is supported by $[0,r]$.
\end{minipage}
\end{align}
\begin{thm}\label{dyzur}
Suppose \eqref{zemanek} holds. Assume that $\cfl\in\bsb(\ell^2(\hhb))$ is subnormal. Then the following conditions hold$:$
\begin{enumerate}
\item[(i)] for every $x \in X$ and $\varrho$-a.e. $w \in W$ the sequence $\big\{\hsf_x^{[n]}(w)\big\}_{n=0}^{\infty}$ is a Stieltjes moment sequence having a unique representing measure $\theta_x^w$,
\item[(ii)] for every $x\in X$ and $\varrho$-a.e. $w\in W$, $\theta_x^w(\rbb_+)=1$ and $\theta_x^w\big(\rbb_+\setminus [0,\|\cfl\|^2]\big)=0$,
\item[(iii)] for every $x\in X$ and $\varrho$-a.e. $w\in W$ we have
\begin{align}
\int_\sigma t\D\theta_x^w=\sum_{y\in \phi^{-1}(\{x\})} |\lambda_y(w)|^2 \theta_y^w (\sigma),\quad \sigma\in \borel{\rbb_+} \label{zgodne}
\end{align}
\end{enumerate}  
\end{thm}
\begin{proof}
First note that by Lemma \ref{potega} we have
\begin{align}
\| \cfl^n \bsf\|^2 &= \sum_{x \in X} \int_W |\lambda_x^{[n]}(w)|^2 |f_{\phi^n(x)}(w)|^2 \D \varrho(w)\notag \\ 
&= \sum_{x \in X} \sum_{y \in \phi^{-n}(\{x\})} \int_W |\lambda_y^{[n]}(w)|^2 |f_x(w)|^2 \D \varrho(w)\notag\\
&= \sum_{x \in X} \int_W \hsf_x^{[n]}(w) |f_x(w)|^2 \D \varrho(w)\notag\\
&= \sum_{x \in X} \int_W \hsf_x^{[n]}(w) |f_x(w)|^2 \D \varrho(w),\quad n\in\zbb_+, \bsf\in\ell^2(\hhb).\label{norma}
\end{align}
Fix $x_0 \in X$ and consider $\bsg\in\ell^2(\hhb)$ such that $g_x = \delta_{x,x_0} g_x$, $x\in X$. Then, by the Lambert theorem,  $\big\{ \| \cfl^n \bsg\|^2\big\}_{n=0}^\infty$ is a Stieltjes moment sequence. Moreover, by \eqref{norma}, we get
\begin{align*}
\| \cfl^n \bsg\|^2 = \int_W \hsf_{x_0}^{[n]}(w) |g_{x_0}(w)|^2 \D \varrho(w), \quad n\in\zbb_+.
\end{align*}
Now, by \eqref{sti}, we have
\begin{align*}
\int_W \bigg(\sum_{m,n \in \zbb_+} \hsf_{x_0}^{[n+m]}(w)  &\alpha(n) \overline{\alpha(m)} \bigg) |g_{x_0}(w)|^2\D \varrho(w) \\
&= \sum_{m,n \in \zbb_+}
\bigg(\int_W  \hsf_{x_0}^{[n+m]}(w)  |g_{x_0}(w)|^2\D \varrho(w)\bigg) \alpha(n) \overline{\alpha(m)} \\
&= \sum_{m,n \in \zbb_+}  \| \cfl^{n+m} \bsg\|^2 \alpha(n) \overline{\alpha(m)}\\  
&\geq 0,\quad \alpha\in \cbb^{(\zbb_+)}.
\end{align*}
In a similar fashion we show that
\begin{align*}
\int_W \bigg(\sum_{m,n \in \zbb_+} \hsf_{x_0}^{[n+m+1]}(w)  \alpha(n) \overline{\alpha(m)} \bigg) |g_{x_0}(w)|^2\D \varrho(w) 
\geq 0,\quad \alpha\in \cbb^{(\zbb_+)}.
\end{align*}
Since $g_{x_0}\in L^2(\varrho)$ may be arbitrary, combining the above inequalities with \eqref{sti}, we deduce that for $\varrho$-a.e. $w\in W$,  $\big\{\hsf_{x_0}^{[n]}(w)\big\}_{n=0}^\infty$ is a Stieltjes moment sequence.

Now, for any fixed $x_0$ we observe that by \eqref{norma} for every $f\in L^2(\varrho)$ we have 
\begin{align*}
\int_W\hsf^{[2(n+1)]}_{x_0}(w)|f(w)|^2\D\varrho (w)&= \|\cfl^{2n+2}\bsf\|^2\\
&\leqslant \|\cfl\|^4\|\cfl^{2n}\bsf\|^2\\ &=\|\cfl\|^4\int_W\hsf^{[2n]}_{x_0}(w)|f(w)|^2\D\varrho(w),\quad n\in\zbb_+,
\end{align*}
where $\bsf\in\ell^2(\hhb)$ is given by $f_x=\delta_{x,x_0}f$. This, according to \eqref{sti}, yields that for $\varrho$-a.e. $w\in W$, $\big\{\hsf^{[n]}_{x_0}(w)\big\}_{n=0}^\infty$ has a unique representing measure $\theta_{x_0}^w$ supported by the interval $[0,\|\cfl\|^2]$. Clearly, for $\varrho$-a.e. $w\in W$, $\theta_{x_0}^w(\rbb_+) = \hsf_{x_0}^{[0]}(w)=1$.

Now, suppose that $x\in X$. Then for $\varrho$-a.e. $w\in W$ we get
\begin{align*}
\int_0^{\infty} t^n \D \theta_{x}^w(t) 
&= \hsf^{[n]}_{x}(w)  
= \sum_{y \in \phi^{-n}(\{x\})} |\lambda_y^{[n]}(w)|^2 \\
&=\sum_{y \in \phi^{-(n-1)}(\phi^{-1}(\{x\}))} |\lambda_y^{[n]}(w)|^2 \\
&= \sum_{z \in \phi^{-1}(\{x\}) } \sum_{y \in \phi^{-(n-1)}(\{z\})}  |\lambda_y^{[n-1]}(w)|^2 |\lambda_{z}(w)|^2 \\
&= \sum_{z \in \phi^{-1}(\{x\}) }  \hsf^{[n-1]}_{z}(w) |\lambda_z(w)|^2\\
&= \int_0^{\infty} t^{n-1} \bigg(\sum_{z \in \phi^{-1}(\{x\})} |\lambda_z(w)|^2\bigg) \D\theta_z^w(t),\quad n\in\nbb.
\end{align*}
This, in view of the fact that $t\D\theta_x^w$ is supported by $[0,\|\cfl\|^2]$, implies that \eqref{zgodne} is satisfied. This completes the proof.
\end{proof}
The representing measures $\theta_x^w$ existing for a subnormal bounded $\cfl$ by the above theorem turn out to be the building blocks for the family $\{\vartheta_x^w\colon x\in X, w\in W\}$.
\begin{thm}\label{konieczny}
Suppose \eqref{zemanek} holds. Assume that $\cfl\in\bsb(\ell^2(\hhb))$ is subnormal. Then there exists a family $\{\vartheta_x^w\colon x\in X, w\in W\}$ of Borel probability measures on $\rbb_+$ such that the following conditions hold:
\begin{enumerate}
\item[(i)] for all $x\in X$ and $\sigma\in\borel{\rbb_+}$ the map $W\ni w\mapsto \vartheta^w_x(\sigma)\in[0,1]$ is $\ascr$-measurable,
\item[(ii)] for all $x\in X$ and $w\in W$ we have $|\lambda_x(w)|^2\vartheta_x^w\ll \vartheta_{\phi(x)}^w$,
\item[(iii)] for every $x\in X$,
\begin{align*}
\Gsf_x =\Gsf_{\phi(x)}\quad  \text{a.e. $[\widehat\varrho_x]$},
\end{align*}
where $\Gsf_x$ is defined by \eqref{tlustazupa} $($see also \eqref{l22} and Lemma \ref{warunekC}$)$.
\end{enumerate}
\end{thm}
\begin{proof}
According to Theorem \ref{dyzur} there exist a set $W_0\in\ascr$ and a family $\{\theta_x^w\colon w \in W_0\}$ of Borel probability measures on $\rbb_+$ such that $\varrho(W\setminus W_0)=0$ and for all $x\in X$ and $w\in W_0$ the condition \eqref{zgodne} holds. Define a family $\{\vartheta_x^w\colon x \in X\}$ of Borel probability measures  by
\begin{align*}
    \vartheta_x^w = \left\{ 
    \begin{array}{cl} 
    \theta_x^w & \text{ for }x\in X, w \in W_0,\\
    \delta_0 & \text{ for }x\in X, w \in W\setminus W_0.
    \end{array}\right.
\end{align*}
In view of (i) of Theorem \ref{dyzur}, the mapping $W\ni w\to \int_{\rbb^+}t^n\D \vartheta_x^w \in\rbb_+$ is $\ascr$-measurable for every $x\in X$, hence applying \cite[Lemma 11]{b-j-j-s-jfa-2015} we get (i). In turn (iii) of Theorem \ref{dyzur} yields (ii). Now, by \eqref{zgodne} and Lemma \ref{warunekC}, for every $x\in X$,  $\varrho$-a.e. $w\in W$ and $\vartheta_x^w$-a.e. $t\in \rbb_+$ we have $\Gsf_x(w,t)=t$, which gives (iii).
\end{proof}
\section{Examples and corollaries}
The operator $M_z$ of multiplication by the independent variable $z$ plays a special role among all multiplication operators. It is easily seen that the weighted bilateral shift operator acting in $\bigoplus_{n=-\infty}^\infty L^2(\varrho)$, the orthogonal sum of $\aleph_0$-copies of $L^2(\varrho)$, with weights being equal to $M_z$ is normal (and thus subnormal). Below we show a more general result stating that for any given $k\in\nbb$ the o-wco $\cfl$ induced by $\phi$ whose graph is a $k$-ary tree (see \cite{b-j-j-s-golf} for terminology) and $\varLambdab$ consists of $\varLambda_x=M_z$ acting in $L^2(\varrho)$ is subnormal.
\begin{exa}
Fix $k\in\nbb$. Let $X=\{1,2,\ldots, k\}^\nbb$ and $\phi\colon X\to X$ be given by $\phi \big(\{\varepsilon_i\}_{i=1}^\infty\big)=\{\varepsilon_{i+1}\}_{i=1}^\infty$. Let $W$ be a compact subset of $\cbb$ and $\varrho$ be a Borel measure on $W$. Finally, let $\hhb=\{\hh_x\colon x\in X\}$ with $\hh_x=L^2(\varrho)$ and let $\varLambdab=\{\varLambda_x\colon x\in X\}$ with $\varLambda_x=M_z$ acting in $L^2(\varrho)$. Then we have $\hsf_x^{[n]}(w)=k^n|w^n|^2$ for every $w\in W$ and $n\in\zbb_+$. This means that $\big\{\hsf_x^{[n]}(w)\big\}_{n=0}^\infty$ is a Stieltjes moment sequence with a unique representing measure $\vartheta_x^w:=\delta_{k|w|^2}$ for every $x\in X$ and $w\in W$. Therefore, conditions \eqref{zemanek} and \eqref{majdak} are satisfied, and $\Gsf_x=\Gsf_{\phi(x)}=k$ for every $x\in X$. According to Theorem \ref{kryterium}, $\cfl$ is subnormal in $\ell^2(\hhb)$.
\end{exa}
A classical weighted unilateral shift in $\ell^2(\zbb_+)$ is subnormal whenever the weights staisfy the well-known Berger-Gellar-Wallen criterion (see \cite{g-w-1970-pja, h-1970-bams}). This can be generalized in the following way.
\begin{exa}
Let $X=\zbb_+$ and $\phi\colon X\to X$ be given by
\begin{align*}
\phi(n)=\left\{ 
\begin{array}{cl} 
0 & \text{if } n=0,\\
n-1 & \text{if } n\in\nbb.
\end{array} \right.    
\end{align*}
Let $(W,\ascr,\varrho)$ be a $\sigma$-finite measure space and let $\hhb=\{\hh_n\colon n\in \zbb_+\}$ with $\hh_n=L^2(\varrho)$. Suppose $\{\lambda_n\}_{n=1}^\infty\subseteq\mscr(\ascr)$ is a family of functions such that for every $w\in W$, the sequence 
\begin{align*}
{\boldsymbol s}^w=(1,|\lambda_1(w)|^2, |\lambda_1(w)\lambda_2(w)|^2, |\lambda_1(w)\lambda_2(w)\lambda_3(w)|^2,\ldots)    
\end{align*} 
is a Stieltjes moment sequence. Set $\lambda_0\equiv 0$. Let $\varLambdab=\{M_{\lambda_n}\colon n\in\zbb_+\}$. Then the o-wco $\cfl$ in $\ell^2(\hhb)$ is subnormal. Indeed, fix $w\in W$. Since ${\boldsymbol s}^w$ is a Stieltjes moment sequence, either $\lambda_k^w=0$ for every $k\in\nbb$ or $\lambda_k^w\neq0$ for every $k\in\nbb$. Let $\theta^w$ be a representing measure of ${\boldsymbol s}^w$. If $\lambda_k^w=0$ for every $k\in\nbb$, then we set $\vartheta_l^w=\delta_0$ for $l\in\zbb_0$. Otherwise, we define a family of probability measures $\{\vartheta_l^w\colon l\in\zbb_+\}$ by
\begin{align*}
\vartheta_l^w(\sigma)=
\left\{
\begin{array}{cl}
\theta^w (\sigma) & \text{if } l=0,\\
\frac{1}{|\lambda_l(w)|^2}\int_\sigma t\D\vartheta_{l-1}^w(t) & \text{if } l\in\nbb,
\end{array} \quad \sigma\in\borel{\rbb_+}.
\right.
\end{align*}
In both cases we see that
\begin{align}\label{procesja}
\int_\sigma t\D\vartheta_l^w(t)=|\lambda_{l+1}(w)|^2 \vartheta_{l+1}^w(\sigma),\quad \sigma\in\borel{\rbb_+},\ l\in\zbb_+.
\end{align}
As a consequence, the family $\{\vartheta_k^w\colon w\in W, k\in\zbb_+\}$ satisfies condition $(\texttt{B})$. Since the mapping $w\mapsto \int_{\rbb_+}t^n\D\theta^w(t)=|\lambda_1(w)\cdots\lambda_{n}(w)|^2$ is $\ascr$-measurable for every $n\in\nbb$, by \cite[Lemma 11]{b-j-j-s-jfa-2015}, the mapping $w\mapsto\vartheta_0^w(\sigma)$ is $\ascr$-measurable for every $\sigma\in\borel{\rbb_+}$. This implies that $\{\vartheta_k^w\colon w\in W, k\in\zbb_+\}$ satisfies condition $(\texttt{A})$. In view of \eqref{procesja}, $\Gsf_l(w,t)=t$ for all $(w,t)\in W\times \rbb_+$ and $l\in\zbb_+$. Therefore, by Theorem \ref{kryterium}, $\cfl$ is subnormal.
\end{exa}
The class of weighted shifts on directed trees with one branching vertex has proven to be a source of interesting results and examples (see \cite{b-j-j-s-jmaa-2013, b-d-j-s-aaa-2013, b-j-j-s-jmaa-2016}). Below we show an example of a subnormal o-wco $\cfl$ induced by a transformation $\phi$ whose graph is composed of a directed tree with one branching vertex and a loop.
\begin{exa}
Fix $k\in\nbb \cup \{\infty\}$. Let $X = \{(0,0)\} \cup \nbb \times \{1,2,\ldots, k\}$. Let $\phi\colon X\to X$ be given by
\begin{align*}
\phi(m,n)=\left\{ 
\begin{array}{cl} 
(0,0) & \text{if } m =0 ,\\
(0,0) & \text{if } m=1 \text { and } n\in\{1,\ldots, k\},\\
(m-1,n) & \text{if } m \geq 2 \text { and } n\in\{1,\ldots, k\}.
\end{array} \right.    
\end{align*}
Let $W$ be a Borel subset of $\cbb$ and $\varrho$ be a Borel measure on $W$. Let $\hhb=\{\hh_x\colon x\in X\}$ with $\hh_x=L^2(\varrho)$. For a given sequence $\{\beta_n\}_{n=1}^{k} \subset \cbb$ such that $\sum_{n=1}^{k} |\beta_n|^2 < \infty$ we define functions $\{\lambda_x\colon x\in X\}\subseteq\mscr(\borel{W})$ by 
\begin{align*}
\lambda_x(w)=\left\{ 
\begin{array}{cl} 
0 & \text{if } x = (0,0),\\
\beta_n  & \text{if } x = (1,n), \\
\sqrt{\sum_{k=1}^n \beta_n^2} & \text{otherwise},
\end{array} \right.\quad w\in W.
\end{align*}
Let $\varLambdab=\{\varLambda_x\colon x\in X\}$ with $\varLambda_x=M_{\lambda_x}$ acting in $L^2(\varrho)$. Finally, let $S=[0,1]$ and $\vartheta_x^w$, $x\in X$ and $w\in W$, be the Lebesgue measure on $S$. Clearly, for every $x \in X$, $\Gsf_x = \sum_{n=1}^{k} \beta_n^2$. Thus by Theorem \ref{kryterium} the operator $\cfl$ is subnormal.
\end{exa}
It is well known that normal operators are, up to a unitary equivalence, multiplication operators. This combined with our criterion can be used to investigate subnormality of $\cfl$ when $\varLambdab$ consists of commuting normal operators.
\begin{exa}
Let $X$ be countable and $\phi\colon X\to X$. Assume that $\hh$ is a separable Hilbert space,  $\hhb=\{\hh_x\colon x\in X\}$ with $\hh_x=\hh$, and  $\varLambdab=\{\varLambda_x\colon x\in X\}\subseteq \bsl(\hh)$ is a family of commuting normal operators. Then there exist a $\sigma$-finite measure space $(W,\ascr,\varrho)$ and a family $\{\lambda_x\colon x\in X\}\subseteq\mscr(\ascr)$ such that for every $x\in X$, $\varLambda_x$ is unitarily equivalent to the operator $M_{\lambda_x}$ of multiplication by $\lambda_x$ acting in $L^2(\varrho)$. Suppose now that there exists a family $\{\vartheta_x^w\colon x\in X, w\in W\}$ of probability measures on a measurable space $(S,\varSigma)$, such that conditions $(\mathtt{A})$ and $(\mathtt{B})$ are satisfied. If for every $x\in X$, $\rho$-a.e. $w\in \cbb$, and $\vartheta_x^w$-a.e. $s\in S$ we have
\begin{align}\label{piwo1}
\sum_{y\in\phi^{-1}(\{x\})}\frac{\D|\lambda_y(w) |^2\vartheta_y^w}{\D\vartheta_x^w}<\infty,
\end{align}
and
\begin{align}\label{piwo2}
\sum_{y\in\phi^{-1}(\{x\})}\frac{\D|\lambda_y(w) |^2\vartheta_y^w}{\D\vartheta_x^w}=\sum_{z\in\phi^{-1}(\{\phi(x)\})}\frac{\D|\lambda_z(w) |^2\vartheta_z^w}{\D\vartheta_{\phi(x)}^w},
\end{align}
then, by applying Theorem \ref{kryterium} and Lemma \ref{warunekC}, we deduce that $\cfl$ acting in $\ell^2(\hhb)$ is subnormal.
\end{exa}
In a similar manner to the case of a family of commuting normal operators we can deal with $\cfl$ induced by $\varLambdab$ consisting of a single subnormal operator.
\begin{exa}
Let $X$ be countable and $\phi\colon X\to X$. Suppose that $\hh$ is a separable Hilbert space and $S$ is a subnormal operator in $\hh$. Let $\hhb=\{\hh_x\colon x\in X\}$ with $\hh_x=\hh$, and $\varLambdab=\{\varLambda_x\colon x\in X\}$ with $\varLambda_x=S$. Since $S$ is subnormal, there exists a Hilbert space $\kk$ and a normal operator $N$ in $\kk$ such that $S\subseteq N$. Let $\kkb=\{\kk_x\colon x\in X\}$ with $\kk_x=\kk$, and $\varLambdab^\prime=\{\varLambda^\prime_x\colon x\in X\}$ with $\varLambda^\prime_x=N$. Clearly, $C_{\phi,\varLambdab^\prime}$ is an extension of $\cfl$. Hence, showing that $C_{\phi,\varLambdab^\prime}$ has a quasinormal extension will yield subnormality of $\cfl$. From this point we can proceed as in the previous example. Assuming that there exists a family $\{\vartheta_x^w\colon x\in X, w\in W\}$ of probability measures on $(S,\varSigma)$ satisfying conditions $(\mathtt{A})$-$(\mathtt{B})$, and conditions \eqref{piwo1} and \eqref{piwo2} for every $x\in X$, $\rho$-a.e. $w\in \cbb$, and $\vartheta_x^w$-a.e. $s\in S$, we can show that $\cfl$ is subnormal.
\end{exa}
The method of proving subnormality via quasinormality and extending the underlying $L^2$-space with help of a family of probability measures has already been used in the context of composition operators (see \cite[Theorem 9]{b-j-j-s-jfa-2015}) and weighted composition operators (see \cite[Theorem 29]{b-j-j-s-wco}). The class of weighted composition operators in $L^2$-spaces over discrete measure spaces is contained in the class of o-wco's (see Remark \ref{wco}). Below we deduce a discrete version of \cite[Theorem 29]{b-j-j-s-wco} from Theorem \ref{kryterium}.
\begin{prop}
Let $(X,2^X,\mu)$ be a discrete measure space, $\phi$ be a self-map of $X$, and $w\colon X\to \cbb$. Suppose that there exists a family $\{Q_x\colon x\in X\}$ of Borel probability measures on $\rbb_+$ such that
\begin{align}\label{CC}
\mu_{x}\int_\sigma t\D Q_{x}(t)=\sum_{y\in\phi^{-1}(\{x\})} Q_y(\sigma) |w(y)|^2\mu_y,\quad \sigma\in \borel{\rbb_+}, x\in X.
\end{align}
and
\begin{align}\label{wolny}
\int_{\rbb_+}t\D Q_x(t)<\infty,\quad x\in X.
\end{align}
Then the weighted composition operator $C_{\phi, w}$ induced by $\phi$ and $w$ is subnormal.
\end{prop}
\begin{proof}
By Proposition \ref{liczaca}, $C_{\phi, w}$ is unitarily equivalent to the weighted composition operator $C_{\phi, \widetilde w}$ in $\ell^2(\nu)$, where $\widetilde w (x)=\sqrt{\frac{\mu_x}{\mu_{\phi(x)}}}\, w(x)$, $x\in X$, and $\nu$ is the counting measure on $2^X$. Obviously, it suffices to prove the subnormality of $C_{\phi, \widetilde w}$ now.

First we note that \eqref{wolny} and \eqref{CC} imply that $\sum_{y\in\phi^{-1}(\{x\})}|w(y)|^2\mu_y<\infty$ for every $x\in X$ (equivalently, the operator $C_{\phi, w}$ is densely defined). Using \eqref{CC} we get
\begin{align*}
\int_\sigma t\D Q_{x}(t)=\sum_{y\in\phi^{-1}(\{x\})} Q_y(\sigma) |\widetilde w (y)|^2,\quad \sigma\in \borel{\rbb_+},\ x\in X.
\end{align*}
This implies that for every $x\in X$ we have $|\widetilde w(x)|^2Q_x\ll Q_{\phi(x)}$ and 
\begin{align}\label{CCprime}
\frac{\D  \big(\sum_{y\in\phi^{-1}(\{x\})} |\widetilde w(y)|^2 Q_y\big)}{\D Q_{x}}=t\quad \text{for $Q_{x}$-a.e. $t\in \rbb_+$.}
\end{align}

Now, we set $W=\{1\}$, $\ascr=\big\{\{1\},\varnothing\big\}$, $\rho (\{1\})=1$, $\lambda_x(1)=\widetilde w (x)$, and $\vartheta_x^1=Q_x$. Then, in view of \eqref{l22}, we have
\begin{align*}
\Gsf_x^1 =\sum_{\substack{y \in \phi^{-1}(\{x\})}}   \frac{\D|\lambda_y(1)|^2 \vartheta_y^1}{\D \vartheta^1_x}
= \frac{\D \big(\sum_{y \in \phi^{-1}(\{x\})} |\widetilde w (y)|^2   Q_y\big)}{\D Q_{x}},\quad x\in X.
\end{align*}
This and \eqref{CCprime} yield $\Gsf_x(t)=t=\Gsf_{\phi(x)}(t)$ for $\vartheta^1_x$-a.e. $t\in\rbb_+$ and every $x\in X$. By Theorem \ref{kryterium} (see also Remark \ref{wco}), the operator $C_{\phi, w^\prime}$ is subnormal which completes the proof. 
\end{proof}
\section{Auxiliary results}
In this section we provide additional results concerning commutativity of a multiplication operators and o-wco's motivated by our preceding considerations. We begin with a commutativity criterion.
\begin{prop}\label{lemKom1}
Let $\{(\varOmega_x,\ascr_x,\mu_x)\colon x\in X\}$ be a family of $\sigma$-finite measure spaces and $\hhb=\{L^2(\mu_x)\colon x\in X\}$. Let $\varGammab=\{\varGamma_x\colon x\in X\}$, with $\varGamma_x\in\mscr(\ascr_x)$, satisfy $M_{\varGammab} \in \bsb(\ell^2(\hhb))$. Let $\varLambdab=\{\varLambda_x\colon x\in X\}$ be a family of operators $\varLambda_x\in\bsl(L^2(\mu_{\phi(x)}), L^2(\mu_x))$. Assume that 
\begin{align}\label{gwiazdka}
M_{\varGamma_x}\varLambda_x\subseteq\varLambda_x M_{\varGamma_{\phi(x)}},\quad x\in X.
\end{align}
Then $M_{\varGammab}\cfl \subseteq \cfl M_{\varGammab}$.
\end{prop}
\begin{proof}
Let $\bsf \in \ell^2(\hhb)$. Since $M_{\varGammab} \in \bsb(\ell^2(\hhb))$, $\bsf\in \dom(M_{\varGammab}\cfl)$ if and only if $f_{\phi(x)}\in\dom(\varLambda_x)$ for every $x\in X$ and
\begin{align*}
    \sum_{x \in X} \int_{\varOmega_x} \big|\big(\varLambda_x f_{\phi(x)}\big)(w)\big|^2 \D \mu_x(w) <\infty.
\end{align*}
On the other hand, $\bsf \in \dom(\cfl M_{\varGammab})$ if and only if $\varGamma_{\phi(x)}f_{\phi(x)}\in\dom(\varLambda_x)$ for every $x\in X$ and
\begin{align}\label{kapusta}
    \sum_{x \in X}\int_{\varOmega_x} \big|\varLambda_x \big(\varGamma_{\phi(x)}f_{\phi(x)}\big)(w)\big|^2 \D \mu_x(w) <\infty.
\end{align}
Now,  if $\bsf \in \dom(M_{\varGammab}\cfl)$, then, by \eqref{gwiazdka}, $\varGamma_{\phi(x)}f_{\phi(x)}\in\dom(\varLambda_x)$ for every $x\in X$. Moreover, since $M_{\varGammab} \in \bsb(\ell^2(\hhb))$ implies that $\varGammab$ is uniformly essentially bounded, we see that 
\begin{align*}
    \sum_{x \in X}\int_{\varOmega_x} \big|\varGamma_x(w) \big(\varLambda_x f_{\phi(x)}\big)(w)\big|^2 \D \mu_x(w) <\infty,
\end{align*}
which, by \eqref{gwiazdka}, implies \eqref{kapusta}. Thus $\dom(M_{\varGammab} \cfl)\subseteq \dom(\cfl M_{\varGammab})$. This and \eqref{gwiazdka} yields
\begin{align*}
\big(M_{\varGammab} \cfl \bsf\big)_x
&=\varGamma_x \big(\cfl \bsf\big)_x =\varGamma_x \varLambda_x f_{\phi(x)}=\varLambda_x \big(\varGamma_{\phi(x)}  f_{\phi(x)}\big)\\
&= \varLambda_x \big(M_{\varGammab} \bsf\big)_{\phi(x)}=\big(\cfl M_{\varGammab} \bsf\big)_x,\quad x\in X,\quad \bsf\in \dom(M_{\varGammab}\cfl),
\end{align*}
which completes the proof.
\end{proof}
\begin{rem}
It is worth noticing that if  $\{(\varOmega_x,\ascr_x,\mu_x)\colon x\in X\}$,  $\hhb$, $\varGammab$, and $\varLambdab$ are as in Proposition \ref{lemKom1}, then $M_{\varGammab} \cfl\subseteq \cfl M_{\varGammab}$ implies $M_{\varGamma_{x}}\varLambda_{x}|_{\dom(C_{\phi, \varLambdab,\phi(x)})}\subseteq\varLambda_{x} M_{\varGamma_{\phi(x)}}$ for every $x\in X$. This can be easily prove by comparing $\big(M_{\varGammab} \cfl \bsf\big)_x$ and $\big(\cfl M_{\varGammab} \bsf\big)_x$ for $\bsf\in\ell^2(\hhb)$ given by $f_{y}=\delta_{y,\phi(x)}g$, where $g\in \dom(C_{\phi,\varLambdab,\phi(x)})$ (see the last part of the proof of Proposition \ref{lemKom1}). 
\end{rem}
In view of our previous investigations it seems natural to ask under what conditions the inclusion in $\cfl M_{\varGammab} \subseteq M_{\varGammab} \cfl$ can be replaced by the equality. Below we propose an answer when $\varLambdab$ consists of multiplication operators.
\begin{prop}\label{quasilem1}
Let $\{(\varOmega,\ascr,\mu_x)\colon x\in X\}$ be a family of $\sigma$-finite measure spaces. Let $\varGammab=\{\varGamma_x\colon x\in X\}\subseteq \mscr(\ascr)$ and $\{\lambda_x\colon x\in X\}\subseteq \mscr(\ascr)$. Suppose that $|\lambda_x|^2\mu_x\ll\mu_{\phi(x)}$ for every $x\in X$. Let $\hhb=\{L^2(\mu_x)\colon x\in X\}$ and $\varLambdab=\{\varLambda_x\colon x\in X\}$ with $\varLambda_x=M_{\lambda_x}\in\bsl(\hh_{\phi(x)}, \hh_x)$. Assume that $H_x:=|\varGamma_x|+\sum_{y\in \phi^{-1}(\{x\})}\frac{\D|\lambda_y|^2\mu_y}{\D\mu_{x}}<\infty$ a.e. $[\mu_{x}]$ for every $x\in X$.  Suppose that $\dom(\cfl)\subseteq \dom(M_{\varGammab})$ and  $\cfl M_{\varGammab} \subseteq M_{\varGammab} \cfl$. Then $\cfl M_{\varGammab}= M_{\varGammab} \cfl$.
\end{prop}
\begin{proof}
We first prove that  $\cfl M_{\varGammab} \subseteq M_{\varGammab} \cfl$ implies
\begin{align}\label{ups2}
\lambda_x \varGamma_x=\lambda_x \varGamma_{\phi(x)}\text{ a.e. $[\mu_x]$},\quad x\in X.
\end{align}
To this end, we fix $x_0\in X$. Since $\mu_{\phi(x_0)}$ is $\sigma$-finite and $|\varGamma_{\phi(x_0)}|+H_{\phi(x_0)}<\infty$ a.e. $[\mu_{\phi(x_0)}]$, using a standard measure-theoretic argument we show that there exists $\{\varOmega_n\}_{n=1}^\infty\subseteq{\ascr}$ such that $\varOmega=\bigcup_{n=1}^\infty \varOmega_n$ and for every $k\in\nbb$ we have $\varOmega_k\subseteq \varOmega_{k+1}$, $\mu_{\phi(x_0)}(\varOmega_k)<\infty$, and $|\varGamma_{\phi(x_0)}|+H_{\phi(x_0)}< k$ on $\varOmega_k$. Now, we consider $\bsf^{(n)}\in \ell^2(\hhb)$, $n\in\nbb$, given by $f_x^{(n)}=\delta_{x,\phi(x_0)}\chi_{\varOmega_n}$, $x\in X$. Then
\begin{align*}
\int_{\varOmega_n}|\varGamma_{\phi(x_0)}|^2\D\mu_{\phi(x_0)} <\infty \text{ and } \int_{\varOmega_n} |\varGamma_{\phi(x_0)}|^2 H_{\phi(x_0)}\D\mu_{\phi(x_0)}<\infty,\quad n\in\nbb,
\end{align*}
which yields $\bsf^{(n)}\in\dom(\cfl M_{\varGammab})$ for every $n\in\nbb$. Consequently, $\bsf^{(n)}\in\dom(M_{\varGammab} \cfl)$ for every $n\in\nbb$. Now, by comparing $M_{\varGammab} \cfl \bsf^{(n)}$ and $\cfl M_{\varGammab} \bsf^{(n)}$, we get $\lambda_x \varGamma_x \chi_{\varOmega_n}= \lambda_x \varGamma_{\phi(x)}\chi_{\varOmega_n}$ a.e. $[\mu_{x}]$ for every $x\in X$ such that $\phi(x)=\phi(x_0)$ and every $n\in\nbb$ (see the last part of the proof of Proposition \ref{lemKom1}). Since $\varOmega=\bigcup_{n=1}^\infty\varOmega_n$, we get the equality in \eqref{ups2} for every $x\in X$ such that $\phi(x)=\phi(x_0)$. Considering all possible choices of $x_0\in X$ we deduce \eqref{ups2}.

Now, let $\bsf \in \dom(M_{\varGammab} \cfl)$. Then $\bsf \in \dom( \cfl)\subseteq \dom(M_{\varGammab})$ and
$$ \sum_{x \in X} \; \int_\varOmega |\lambda_x\varGamma_x|^2 |f_{\phi(x)}|^2 \D \mu_x < \infty.$$
This combined with \eqref{ups2} imply that $\bsf\in \dom(\cfl M_{\varGammab})$. Hence $\dom(M_{\varGammab} \cfl)=\dom(\cfl M_{\varGammab})$ which, in view of $\cfl M_{\varGammab} \subseteq M_{\varGammab} \cfl$, proves the claim.
\end{proof}
\begin{cor}\label{tikitaki}
Let $\{(\varOmega, \bscr,\mu_x)\colon x\in X\}$ be a family of $\sigma$-finite measure spaces. Let $\{\xi_x\colon x\in X\}\subseteq \mscr(\bscr)$. Suppose that $|\xi_x|^2\mu_x\ll\mu_{\phi(x)}$ for every $x\in X$. Let $\hhb=\{L^2(\mu_x)\colon x\in X\}$. Assume that $\varGammab=\{\varGamma_x\colon x\in X\}$ is a family of functions $\varGamma_x\in\mscr(\bscr)$ such that $\varGamma_x=\varGamma_{\phi(x)}$ a.e. $[\mu_x]$ and $z_0-M_{\varGammab}$ is an invertible operator in $\ell^2(\hhb)$ for some $z_0\in \cbb$. Let $\varXib=\{\varXi_x\colon x\in X\}$ with $\varXi_x=M_{\xi_x}\in\bsl(\hh_{\phi(x)}, \hh_x)$, $x\in X$. Suppose that $C_{\phi, \varXib}$ and $M_{\varGammab}$ are densely defined, and $\dom(C_{\phi, \varXib})\subseteq \dom(M_{\varGammab})$. Then $C_{\phi, \varXib} M_{\varGammab}= M_{\varGammab} C_{\phi, \varXib}$.
\end{cor}
\begin{proof}
Since $z_0-M_{\varGammab}$ is invertible, $z_0$ does not belong to the essential range of any $\varGamma_x$, $x\in X$, and $\big(z_0-M_{\varGammab}\big)^{-1}=M_{\boldsymbol{\varDelta}}$ where $\varDeltab=\{\varDelta_x\colon x\in X\}$ with $\varDelta_x:=(z_0-\varGamma_x)^{-1}$ (note that $\varGamma_x<\infty$ a.e. $[\mu_x]$ because $M_{\varGammab}$ is densely defined). Then $\varDelta_x=\varDelta_{\phi(x)}$ a.e. $[\mu_x]$ for every $x\in X$ which means that $M_{\varDelta_x}\varXi_x\subseteq\varXi_x M_{\varDelta_{\phi(x)}}, x\in X$. Consequently, by Proposition \ref{lemKom1}, we get $M_{\varDeltab}C_{\phi, \varXib}\subseteq C_{\phi, \varXib} M_{\varDeltab}$. This in turn implies that $C_{\phi, \varXib} M_{\varGammab}\subseteq M_{\varGammab}C_{\phi, \varXib}$. Dense definiteness of $C_{\phi, \varXib}$ yields $\sum_{y\in\phi^{-1}(\{x\})}\frac{\D|\xi_y|^2\mu_y}{\D\mu_x}<\infty$ a.e. $[\mu_x]$ for every $x\in X$. Hence, by Proposition \ref{quasilem1}, we show that $C_{\phi, \varXib} M_{\varGammab}$ and $M_{\varGammab}C_{\phi, \varXib}$ coincide.
\end{proof}
As a byproduct of Corollary \ref{tikitaki} we get another proof of Theorem \ref{kryterium}.
\begin{proof}[Second proof of Theorem \ref{kryterium}]
We apply Corollary \ref{tikitaki} with $(\varOmega,\bscr,\mu_x)=(W\times S, \ascr\otimes\varSigma,\widehat\varrho_x)$, $\varGamma_x=\sqrt{\Gsf_x}$, $\varXi_x=\widehat\varLambda_x$, and any $z_0\in \cbb$ with non-zero imaginary part. Clearly, since $M_{\varGammab}=M_{\sqrt{\Gsf}}$ is selfadjoint, $z_0-M_{\varGammab}$ is invertible. Moreover, $\dom(C_{\phi, \varXib})=\dom(C_{\phi,\widehat\varLambdab})=\dom(|C_{\phi,\widehat\varLambdab}|)=\dom(M_{\sqrt{\Gsf}})=\dom(M_{\varGammab})$ by Proposition \ref{dziekan}. Therefore, applying Corollary \ref{tikitaki} we get $C_{\phi,\widehat\varLambdab} M_{\sqrt\Gsf}= M_{\sqrt\Gsf} C_{\phi,\widehat\varLambdab}$, which implies that $C_{\phi,\widehat\varLambdab} M_{\Gsf}= M_{\Gsf} C_{\phi,\widehat\varLambdab}$. Applying Proposition \ref{dziekan} again, we obtain $C_{\phi,\widehat\varLambdab}^* C_{\phi,\widehat\varLambdab} C_{\phi,\widehat\varLambdab}=C_{\phi,\widehat\varLambdab} C_{\phi,\widehat\varLambdab}^* C_{\phi,\widehat\varLambdab}$. In view of \eqref{QQQ} the proof is complete.
\end{proof}
The final important observation is that the condition appearing in Theorem \ref{kryterium} is necessary for the quasinormality of $C_{\phi, \widehat\varLambda}$.
\begin{prop}
Assume \eqref{zemanek} and \eqref{majdak}. If $C_{\phi, \widehat \varLambdab}$ is quasinormal, then
\begin{align*}
\lambda_x\Gsf_{\phi(x)}=\lambda_x \Gsf_x\quad \text{a.e. $[\widehat\varrho_{x}]$ for every $x\in X$.}
\end{align*}
\end{prop}
\begin{proof}
Since $C_{\phi, \widehat \varLambdab}$ is densely defined $\Gsf_x<\infty$ a.e. $\widehat\varrho_{x}$. Now it suffices to argue as in the proof of Proposition \ref{quasilem1} to get $\lambda_x \Gsf_x = \lambda_x\Gsf_{\phi(x)}$ a.e. $[\widehat\varrho_x]$ (cf. \eqref{ups2}). 
\end{proof}

\section{Acknowledgments}
The research of the second and third authors was supported by the Ministry of Science and Higher Education (MNiSW) of Poland.
\bibliographystyle{amsalpha}

\end{document}